\PassOptionsToPackage{nameinlink}{cleveref}
\documentclass[a4paper,UKenglish,cleveref]{lipics-v2019}

\usepackage{tikz}
\usetikzlibrary{arrows,shapes,positioning}

\Crefname{theorem}{Theorem}{Theorems}
\Crefname{definition}{Definition}{Definitions}

\hypersetup{
    colorlinks=true,
    linkcolor=blue,
    filecolor=magenta,      
    urlcolor=cyan,
}

\nolinenumbers

\title{Simple Analysis of Johnson-Lindenstrauss Transform under Neuroscience Constraints} 

\titlerunning{Dummy short title}

\author{Maciej Skorski}{University of Luxembourg}{}{}{}

\authorrunning{M. Skorski}

\Copyright{M. Skorski}



\category{}

\relatedversion{}

\supplement{}

\acknowledgements{ }

\EventEditors{John Q. Open and Joan R. Access}
\EventNoEds{1}
\EventLongTitle{42nd Conference on Very Important Topics (CVIT 2016)}
\EventShortTitle{.}
\EventAcronym{CVIT}
\EventYear{2016}
\EventDate{December 24--27, 2016}
\EventLocation{Little Whinging, United Kingdom}
\EventLogo{}
\SeriesVolume{}
\ArticleNo{}

\begin{document}

\maketitle

\begin{abstract}
The paper re-analyzes a version of the celebrated Johnson-Lindenstrauss Lemma, in which matrices are subjected to constraints
that naturally emerge from neuroscience applications: a) sparsity and b) sign-consistency.
This particular variant was studied first by Allen-Zhu, Gelashvili, Micali, Shavit and more recently by Jagadeesan (RANDOM'19).

The contribution of this work is a novel proof, which in contrast to previous works a) uses the modern probability toolkit, particularly basics of sub-gaussian and sub-gamma estimates b) is self-contained, with no dependencies on subtle third-party results c) offers explicit constants.

At the heart of our proof is a novel variant of Hanson-Wright Lemma (on concentration of quadratic forms). Of independent interest are also auxiliary facts on sub-gaussian random variables.
\end{abstract}

\section{Introduction}

\subsection{Johnson-Lindenstrauss Transform}

The point of departure for our discussion is the celebrated result due to Johnson and Lindenstrauss~\cite{johnson1984extensions}, which shows that any high-dimensional data can be compressed to a much lower dimension, while almost preserving the original geometry (distances).

The JL Lemma is widely used in data analysis as a preprocessing step, to reduce the size of data to be feed into algorithms.
Over the years it has found numerous applications across many different fields, for example in streaming and search algorithms~~\cite{alon1999space,ailon2009fast}, fast approximation algorithms for statistical and linear algebra~\cite{clarkson2009numerical,clarkson2008tighter,sarlos2006improved}, algorithms for computational biology~\cite{bertoni2005random} and even privacy~\cite{blocki2012johnson}. Formally, the lemma can be stated as follows
\begin{lemma}[JL Lemma~\cite{johnson1984extensions}]
For every integers $m,n>1$, subset $\mathcal{X}\subset\mathbb{R}^{m}$ of cardinality $\#\mathcal{X} = n$, and $\epsilon \in (0,1)$ there exists a $d\times m$ real matrix $A$, with $d = O(\epsilon^{-2}\log n)$, such that
\begin{align}\label{eq:jl}
\forall x,x' \in \mathcal{X}:\quad (1-\epsilon)\|x-x'\|_2 \leqslant \| Ax - Ax'\|_2 \leqslant (1+\epsilon)\|x-x'\|_2.
\end{align}  
\end{lemma}
We note that the above relation of distortion $\epsilon$ to the dimension $d$ given is known to be optimal~\cite{alon2003problems,jayram2013optimal}; this however may change when extra conditions are imposed on $A$.

Results of this sort are proven by the probabilistic method. It suffices to establish the above for all pairs $(x,0)$ with high probability, when $A$ is sampled from an appropriate distribution. More precisely, the following is called \emph{Distributional JL Lemma}
\begin{align}\label{eq:djl}
 (1-\epsilon)\|x\|_2 \leqslant \| Ax \|_2 \leqslant (1+\epsilon)\|x\|_2\quad \text{w.p.}\ 1-\delta \text{ over }A\sim\mathcal{A}
\end{align}  
where $d = \mathrm{poly}(\epsilon^{-1},\log(1/\delta))$ may depend on the structure of $\mathcal{A}$. Then \Cref{eq:jl} follows by
applying the above to $x-x'$ in place of $x$ for each pair $x,x'\in\mathcal{X}$, setting $\delta < 1/\binom{n}{2}$ and taking the union bound over all pairs. As for the distribution $\mathcal{A}$ we note that already fairly simple constructions, for example $A$ filled with Rademacher or Gaussian entries do the job.

\subsection{JL Transform with Neuro-Science Constraints}

Given the wide range of applications, one is often interested in imposing additional requirements on $A$.
For example, from the algorithmic perspective it is desirable to establish sparsity of $A$~\cite{kane2014sparser}, in order to compute projections faster. In this paper we however focus on a more subtle constraint, inspired by  neuroscience and studied
in recent works~\cite{jagadeesan2019simple,allen2014sparse}.
\begin{definition}[Sparse Sign-Consistent Matrix~\cite{jagadeesan2019simple,allen2014sparse}]\label{def:matrix_structure}
A matrix is called $p$-sparse and sign-consistent when a) all but a $p$-fraction of entries in each column are zero
b) entries in the same column are of same sign.
\end{definition}
The conditions in \Cref{def:matrix_structure} and the context of JL transform are rooted in how brain operates. More specifically, 
think of $n$ input (presynaptic) neurons communicating to $d$ output (postsynaptic) neurons through synaptic connections.
First, synaptic connections are very sparse (billions of neurons but only few thousands synapses per average neuron~\cite{drachman2005we}); denote the transpose of connectivity matrix by $A$, then we obtain that $A$ is column-sparse.
. Second, a neuron triggers an action when the internal charge of the cell exceeds a certain threshold;
and the internal charge by superposition of received signals (potentials) from input neurons~\cite{purves2008neuroscience}; this aggregation of potential $x$ can be modeled by the multiplication $Ax$. Third, potentials can increase (excitatation) or decrease (inhibition) the likelihood of action, depending on the kind of released chemical (transmitter); at a given point of time we may (simplistically) assume each neuron is releasing either increasing or decreasing signal to its neighbors, which means that columns of $A$ are sign-consistent. Finally, we have empirical evidence that similarity structure of the information flowing through the network is preserved in brains of humans and animals~\cite{kiani2007object,ganguli2012compressed}; a plausible explanation seems to be exactly the low-distortion of distances guaranteed by the JL lemma, as discussed in~\cite{jagadeesan2019simple,allen2014sparse}.

Having explained the motivation, we are now ready to formulate the JL Lemma for sparse, sign-consistent matrices. The following theorem summarizes the prior works
\begin{theorem}[Sparse Sign-Consistent Distributional JL~\cite{jagadeesan2019simple,allen2014sparse}]\label{thm:djl}
For every integer $d>1$ every real numbers $\epsilon,\delta>0$, there exist a sampling distribution $\mathcal{A}$ over sparse sign-consistent matrices $d\times m$ such that \eqref{eq:djl} holds, with $d = O(\epsilon^{-2}\log^2(1/\delta))$ and sparsity $p=O(\epsilon/\log(1/\delta))$.
\end{theorem}
\begin{note}
Extra tradeoff, but is it really worth doing it?
\end{note}

The remainder of this paper is dedicated to give an alternative proof of \Cref{thm:djl}. The main motivation is that the prior proofs~\cite{jagadeesan2019simple,allen2014sparse} are a) long and hard to follow, in that they either use extremely complicated combinatorics~\cite{allen2014sparse} or invoke deep third-party probability results~\cite{jagadeesan2019simple} (such as exotic bounds on moment of iid sums~\cite{latala1997estimation}, or extensions of Khintchine's inequality~\cite{hitczenko1997moment}) b) do not build on modern probability tools, in that they involve painful estimates of moments as opposed to usual proofs of JL lemmas which rely on concentration inequalities derived via MGF~\cite{vershynin2017four,boucheron2013concentration}) c) do not offer explicit constants, which makes them less usable in practice (e.g. for statistical or machine learning software).

The main source of difficulty seems to be the row-independence property which makes indeed possible to give a "few-liner" argument for the standard JL lemma~\cite{vershynin2017four,boucheron2013concentration}, but breaks in case of sign-constraint matrices. This seems to motivate the authors~\cite{jagadeesan2019simple,allen2014sparse} to take a different route and estimate the moments directly. This discussion leads to

\begin{quote}
\textbf{Challenge } Prove sparse sign-consistent JL, relying on standard estimates of MGF.
\end{quote}


\subsection{Contribution}

\subsubsection{Results}

Our main ingredient is a general result of independent interest, a version of Hanson-Wright Lemma~\cite{hanson1971bound}. This result differs from other works, in that it allows certain dependencies between entries of a matrix. It works well in case of Sparse Sign-Consistent JL Lemma.

\begin{theorem}[Version of Hanson-Wright Lemma]\label{thm:our_main}
Let $Q_{j,j'}$ be any random variables, and $\sigma_i$ be independent (also of $Q_{j,j'}$) Rademacher random variables. Consider the quadratic form
\begin{align}
E(x) = \sum_{j\not=j'} Q_{j,j'}\sigma_j\sigma_{j'}x_j x_{j'} 
\end{align}
and define
\begin{align}
v \triangleq \sup_{x:\|x\|_2\leqslant 1} \max_j \mathbf{E}(\sum_{j'} Q_{j,j'}^2 x_{j'}^2)^{1/2}.
\end{align}
then we have the following upper-tail inequality (and same for the lower tail)
\begin{align}
\Pr[E(x)>\epsilon] \leqslant \exp(-\min(\epsilon^2/128v^2,\epsilon/16v)),\quad \epsilon>0.
\end{align}
When $Q_{j,j'}$ are normalized so that $\mathbf{E}Q_{j,j'}^2\leqslant q^2$ for each $j,j'$ then we have $v\leqslant q$.
\end{theorem}
Armed with this general result we relatively easily conclude
\begin{corollary}[Sparse Sign-Consistent JL Lemma]\label{cor:our_main} 
\Cref{thm:djl} holds with parameters sparsity $p = \frac{\epsilon } {16\sqrt{2}\log(2/\delta)}$ and dimension $d = \lceil \frac{ 512 \log^2(2/\delta)}{\epsilon^2}  \rceil$.
\end{corollary}

\subsubsection{Proof Outline}

Before getting into details, we elaborate on techniques used in our proof of \Cref{thm:our_main}. 
The corner step is an application of the decoupling inequality, which allows us to consider a bilinear form in Rademacher variables.
Since the matrix rows are not independent, we estimate the moment generating function (MGF) conditionally on fixed values of matrix rows.
By leveraging convexity, we reduce the problem to estimating the MGF of the square of linear form, corresponding to an individual row.
Once we establish that this form is sub-gaussian, the bound follows (we utilize the tail integration formula).
The final tail bound follows by an argument similar to the one used in Bernstein's inequality (we obtain sub-gamma tails).
An outline is provided in~\Cref{fig:outline}.
We stress that we use only basic estimates on MGFs, taught in modern probability courses~\cite{vershynin2018high}, with the intent to make the result more accessible.
\begin{figure}[h!]
\centering
\begin{tikzpicture}
\node[draw,rectangle] (quad_form) {$MGF( \sum_{j,j'} Q_{j,j'}\sigma_{j}\sigma_{j'}x_j x_{j'} ) \leqslant^{?}$ };
\node[draw,rectangle,below=1cm of quad_form] (bilinear) { $MGF( \sum_{j} x_j\sigma_j \underbrace{ (\sum_{j'}Q_{j,j'}x_{j'}\sigma'_{j'} )}_{Y_j} \leqslant^{?}$ ) };
\node[draw,rectangle,below=1cm of bilinear] (squares_linear) { $MGF( \sum_j x_j^2 Y_j^2)\leqslant^{?}$ };
\node[draw,rectangle,below=1cm of squares_linear] (square_linear) { $\max_j MGF( Y_j^2) \leqslant^{?} $ };
\node[draw,rectangle,below=1cm of square_linear] (tail_linear) { $\Pr[ Y_j > \epsilon] \leqslant^{?} $ };
\node[draw,rectangle,below=1cm of tail_linear] (mgf_linear) {$\mathbf{E}Q_{j,j'}^2\leqslant^{?}$};
\draw[->] (quad_form)--(bilinear) node[midway,right] {decoupling};
\draw[->] (bilinear)--(squares_linear) node[midway,right] {row conditioning ($Y_j$), MGF of $\sigma_i$}; 
\draw[->] (squares_linear)--(square_linear) node[midway,right] {Jensen's inequality}; 
\draw[->] (square_linear)--(tail_linear) node[midway,right] {tail integration};
\draw[->] (tail_linear)--(mgf_linear) node[midway,right] {sub-gaussian properties};
\end{tikzpicture}
\caption{The proof outline. Arrows follow the reduction, boxes illustrate the proof steps.}
\label{fig:outline}
\end{figure}
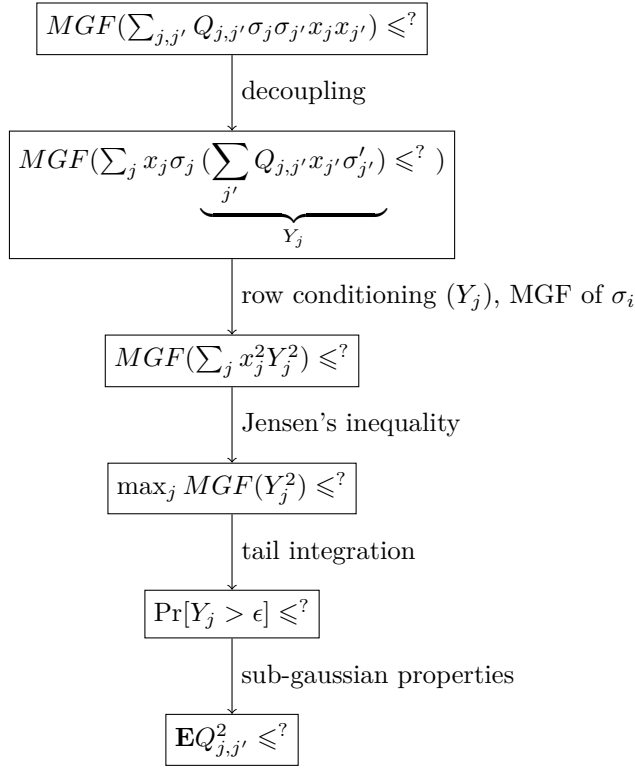

\section{Prelimaries}

We need to establish some terminology. Rademacher random variable takes values $\pm 1$ with probability $1/2$.
The moment generating function of a random variable $X$ is defined as $MGF_X(t) = \mathbf{E}\exp(tX)$. Below we prove some auxiliary results.

\subsection{Computing Expectation}

Below we present an extension of the fact used usually for moments ($h(x)=x$ or $h(x)=x^{k}$)
\begin{lemma}[Expectation by tail integration]\label{lem:tail_integral}
For a non-negative r.v. $X$ and a monotone function $h$ such that $h(0)=0$ it holds that
\begin{align*}
\mathbf{E}h(X) = \int_{0}^{\infty}  \frac{\partial h(y)}{\partial y}\cdot \mathbf{P}(X>y) \mbox{d} y.
\end{align*}
\end{lemma}
\begin{proof}
Applying Fubinni's theorem to justify the change of integrals we otabin
\begin{align*}
\mathbf{E}h(X) &= \int h(x) \mbox{d}\mathbf{P}(X\leqslant x) = \int_{0}^{\infty} \int_{0}^{x}\frac{\partial h(y)}{\partial y} \mbox{d}y\,\mbox{d} \mathbf{P}(X\leqslant x) \\
& =  \int_{0}^{\infty} \frac{\partial h(y)}{\partial y} \mbox{d}y \int_{y}^{\infty}\mbox{d} \mathbf{P}(X\leqslant x) = \int_{0}^{\infty}  \frac{\partial h(y)}{\partial y}\cdot \mathbf{P}(X>y) \mbox{d} y.
\end{align*}
\end{proof}

\subsection{Sub-Gaussian Distributions}

Sub-gaussian distributions have tails lighter than gaussian and their MGFs enjoy several nice properties.
Below we discuss some of them, and for a more complete treatment refer to~\cite{vershynin2018high}.

\begin{definition}[Sub-gaussian random variables]\label{def:subgauss}
A random variable $X$ is called sub-gaussian with variance factor $v^2$, when 
$\mathbf{E}\exp(t X) \leqslant \exp(t^2v^2/2)$ for every real number $t$.
\end{definition}

\begin{lemma}[Sub-gaussian tail]\label{lemma:subgauss_tail}
If $X$ is sub-gaussian with variance factor $v^2$ then
\begin{align*}
\Pr[|X|>\epsilon] \leqslant \exp(-\epsilon^2/2v^2)
\end{align*}
\end{lemma}
\begin{proof}
We use Chernoff's method: by Markov's inequality $\Pr[X>\epsilon]\leqslant \exp(-t\epsilon)\mathbf{E}\exp(tX)$.
By the assumption, this is at most $\exp(t^2v^2/2 - t\epsilon)$. We optimize $t$ by choosing $t=\epsilon/v^2$.
\end{proof}

\begin{lemma}[Sub-gaussian norm]\label{lemma:subg_norm}
Given a random variable $X$ define
\begin{align*}
\|X\|_{sG} \triangleq \inf\{v>0: \mathbf{E}\exp(tX) \leqslant \exp(t^2v^2/2)\quad\text{holds for every } t\},
\end{align*}
the best constant $v$ which satisfies \Cref{def:subgauss}. 
Then we have
\begin{enumerate}[(i)]
\item $\|\cdot\|_{sG}$ is a norm; in particular $\|\sum_i X_i \|_{sG} \leqslant \sum_i\|X_i\|_{sG}$ for any $X_i$
\item we have $\|\sum_i X_i \|_{sG}^2 \leqslant \sum_i\|X_i\|_{sG}^2$ for independent $X_i$.
\end{enumerate}
\end{lemma}
\begin{proof}
Suppose that $\mathbf{E}\exp(tX_i) \leqslant \exp(t^2v_i^2/2)$. Define $\theta_i = v_i / \sum_j v_j$. Then we have
\begin{align*}
\mathbf{E}\exp(t (X_1+\ldots + X_n)) &=\mathbf{E}\exp(t \sum_i\theta_i \cdot (X_i/\theta_i)) \\ 
& \leqslant^{(a)} \sum_i \theta_i \mathbf{E}\exp(t X_i/\theta_i) \\
& \leqslant^{(b)} \sum_i \theta_i \exp(t^2 (\sum_i v_i)^2/2) = \exp(t^2(\sum_i v_i)^2/2)
\end{align*}
where (a) follows by Jensen's inequality and (b) by the assumption on $v_i$.
It follows that $X=\sum_iX_i$ is sub-gaussian with factor $v = \sum_i v_i$, which proves that 
$\|\cdot\|_{sG}$ is sub-additive. Since for any $\theta>0$, by definition, $\|\theta X\|_{sG} = \theta \|X\|_{sG}$, it is a norm which proves (i). 

To prove (ii) it suffices to observe that by independence and the assumption on $v_i$
\begin{align*}
\mathbf{E}\exp(t (X_1+\ldots + X_n)) \leqslant \prod_i\mathbf{E}\exp(tX_i) \leqslant \prod_i \exp(t^2v_i^2/2) = \exp(t^2\sum_iv_i^2/2)
\end{align*}
so that $\mathbf{E}\exp(t (X_1+\ldots + X_n))  \leqslant \exp(t^2v^2/2)$ with $v^2 = \sum_i v_i^2$.
\end{proof}

\begin{remark}[Simpler norm definition]
The typical textbook approach to define the norm uses a different characterization of the sub-gaussian property~\cite{vershynin2018high}, namely
$\|X\|_{Sg} \triangleq \inf\{ t: \mathbf{E}\exp(X^2/t^2) \leqslant 2 \}$ (a special case of Orlicz norms). This norm is equivalent with our definition up to a multiplicative constant, but the proof is much more tricky.
\end{remark}

\begin{lemma}[Square of sub-gaussian]\label{lemma:subgauss_square}
Let $X$ be sub-gaussian with variance factor $v^2$. Then
\begin{align*}
\mathbf{E}\exp(t X^2) \leqslant \exp\left( \frac{t\cdot 2v^2}{1- t\cdot 2v^2}\right), \quad 2 t v^2 < 1.
\end{align*}
\end{lemma}
\begin{remark}[Simpler derivation] Similar results follow by somewhat painful estimation of moments. Our technique leads to a simple proof which appears to be novel (see~\cite{vershynin2018high,honorio2014tight}).
\end{remark}

\begin{proof}
By  \Cref{lem:tail_integral} applied to $|X|$ and $h(x) = \exp(tx^2)$ we have
\begin{align*}
\mathbf{E}(\exp(t X^2)-1) &= \int 2 t x\exp(t x^2)\mathbf{P}(|X|>x)\mbox{d}x \\
& \leqslant^{(a)} \int 2 t x\exp((t-1/2v^2) x^2)\mbox{d}x \\
& =^{(b)} 2t / (1/2v^2-t),\quad t < 1/2v^2 \\
&\leqslant^{(d)} \exp(2t v^2 / (1-2v^2 t))-1, \quad t < 1/2v^2.
\end{align*}
where a) is due to \Cref{lemma:subgauss_tail} and b) due to the inequality $u\leqslant \exp(u)-1$.
\end{proof}

\subsection{Sub-Gamma Distributions}

In case of sub-gamma distributions the MGF exists only up to a certain point. 
Below we review their tail behavior, referring to~\cite{boucheron2005moment} for a more exhaustive discussion.

\begin{definition}[Sub-gamma distribution]\label{def:sub-gamma}
A random variable $X$ is sub-gamma with variance factor $v^2>0$ and scale $c>0$ when
\begin{align*}
\mathbf{E}\exp(tX) \leqslant \exp\left(\frac{v^2t^2}{2(1-t c)}\right),\quad  0<|t|<1/c.
\end{align*}
\end{definition}
The following may be seen as a variant of Bernstein's inequality.
\begin{lemma}[Tails of sub-gamma distributions]\label{lemma:sugbamma_tails}
For $X$ as in \Cref{def:sub-gamma} it holds that
\begin{align}
\Pr[X>\epsilon] \leqslant \exp(-\min(\epsilon^2/4v^2,\epsilon/4c),\quad \epsilon > 0.
\end{align}
and same holds for the lower tail $\Pr[X<-\epsilon]$.
\end{lemma}
\begin{remark}
With more work (finding the exact solution when optimizing Chernoff's inequality) one can prove the bound $\exp\left(-\frac{\epsilon^2}{2(v^2+b\epsilon)}\right)$.
\end{remark}
\begin{proof}
For $|t|\leqslant 1/2c$ we have $\mathbf{E}\exp(tX)\leqslant \exp(v^2 t^2)$. By Markov's inequality
$\Pr[X>\epsilon]\leqslant \exp(v^2 t^2 - t\epsilon)$, which we optimize over $t$. The global minimum is  
$t = \epsilon / 2v^2$ with the value of $\exp(-\epsilon^2/4v^2)$. When $\epsilon/2v^2 > 1/2c$, we use
$t = 1/2c$ so that $\exp(v^2t^2-t\epsilon)\leqslant \exp(t\epsilon/2-t\epsilon)=\exp(-\epsilon/4c)$. 
Replacing $X$ with $-X$ gives same for the lower tail.
\end{proof}

\section{Proof of Sparse Sign-Consistent JL Lemma}

\subsection{Reduction to Quadratic Form Concentration}

We construct the sampling distribution for matrix $A$ as in~\cite{jagadeesan2019simple}. Let $\sigma_j$ for $j=1,\ldots,n$ be independent Rademachers. Let $s \leqslant d$ be an integer, and let $\eta  = \eta_{i,j}$ be a $d\times m$ boolean matrix
chosen so that for each column $j\in [m]$ we independently select randomly $s$ out of $d$ places and set them to be $1$, declaring zero on the remaining $d-s$ entries. Now let
\begin{align*}
    A_{i,j} = \eta_{i,j} \sigma_j / \sqrt{s}
\end{align*}
By construction $A$ is $p$-sparse with $p=s/d$ and sign-consistent. Since $A$ is linear, it suffices to prove \eqref{eq:djl} for unit vectors $x$, e.g. $\|x\|_2 \triangleq (\sum_j x_j^2)^{1/2} = 1$. We have to show
\begin{align}
|E(x)|\leqslant \epsilon\ \text{w.p.}\ 1-\delta,\quad  E(x)\triangleq    \|A x\|_2^2 -1 
\end{align}
for every unit $x\in\mathbb{R}^m$. Observe that due to our assumptions (definition of $\eta$ and $x$) we have
\begin{align*}
E(x)=     \| (s^{-1/2}\sum_j \eta_{i,j}\sigma_j x_j )_i \|_2^2-1= \frac{1}{s} \sum_i \sum_{j\not=j'} \eta_{i,j}\eta_{i,j'} \sigma_j \sigma_{j'} x_j x_{j'}
\end{align*}
(diagonal cases $j=j'$ aggregated to 1). Equivalently we write (the same form as in~\cite{jagadeesan2019simple})
\begin{align*}
    E(x) = \sum_{j\not=j'} Q_{j,j'} \sigma_j\sigma_{j'} x_j x_{j'},\quad Q_{j,j'}\triangleq \frac{1}{s}\sum_{i} \eta_{i,j}\eta_{i,j'}.
\end{align*}
Note, which will be essential throughout our proof, that $Q_{j,j'}$ are independent of $\sigma_j$.

\subsection{Concluding Sparse Sign-Consistent JL Lemma}

We shall use \Cref{thm:our_main} to conclude~\Cref{cor:our_main}. Recall that
$Q_{j,j'}= \frac{1}{s}\sum_{i} \eta_{i,j}\eta_{i,j'}$. Let $I$ be the set of $i$ such that $\eta_{i,j} = 1$. Note that $|I|\leqslant s$, we have
\begin{align*}
\mathbf{E} [Q_{j,j'}^2 | I] \leqslant s^{-2}\mathbf{E}(\sum_{i\in I}\eta_{i,j'})^2
\end{align*}
By construction we have that $\eta_{i,j'}$ for different $i$ (along the $j'$-th column) are negatively correlated.
Thus we get
\begin{align*}
\mathbf{E} [Q_{j,j'}^2 | I] \leqslant s^{-2}(sp + s(s-1)p^2) = p^2 + s^{-1}p(1-p)
\end{align*}
and the same bound is valid for $\mathbf{E} Q_{j,j'}^2  $ by taking the expectation over $I$.
It follows that we can take $v^2 = 2p^2$ provided that $p \geqslant 1/s$. We need $p$ such that
$\epsilon = 16\sqrt{2} p  \log(1/\delta)$ to get $\Pr[X>\epsilon] \leqslant \delta$ and $\Pr[X<-\epsilon] \leqslant \delta$.
\Cref{cor:our_main} now follows since $d = s/p$. 

\section{Proof of Main Theorem}

\subsection{Moment Generating Function of Quadratic Form}
We aim to bound the moment generating function of $E(x)$, that is
\begin{align*}
\mathbf{E}\exp(t E(x)) =  \mathbf{E}\exp(t\sum_{j\not=j'}Q_{j,j'} x_j x_{j'} \sigma_j\sigma_{j'})
\end{align*}
The tail bound will follow by Markov's inequality and optimizing over $t$ (Chernoff's method).

\subsection{Decoupling}

By the widely known decoupling inequality for off-diagonal matrices (cf.~\cite{vershynin2011simple}) we obtain
\begin{proposition}
If $\sigma'_j$ are independent Rademachers then we have
\begin{align}
\mathbf{E}\exp(t E(x)) \leqslant \mathbf{E}\exp(4 t\sum_{j\not=j'}Q_{j,j'} x_j x_{j'} \sigma'_j\sigma_{j'})
\end{align}
\end{proposition}
\begin{proof}
The claim follows by the decoupling result applied to the convex function $x\to \exp(t x)$ and the quadratic form $E(x)$ conditionally on $Q_{j,j'}$ (which is independent of $\sigma_{j}$). This shows
\begin{align*}
\mathbf{E}[\exp(t E(x)|Q_{j,j'}]) \leqslant \mathbf{E}[\exp(4 t\sum_{j\not=j'}Q_{j,j'} x_j x_{j'} \sigma'_j\sigma_{j'})| (Q_{j,j'})_{j,j'}]
\end{align*}
and the result follows by taking the expectation over all $Q_{j,j'}$. 
\end{proof}

\subsection{Reduction to Linear Form}

We proceed further by rewriting the decoupled form as the double sum
\begin{align}\label{eq:double_sum}
\sum_{j\not=j'}Q_{j,j'} x_j x_{j'} \sigma'_j\sigma_{j'} = \sum_{j} Y_{j} x_{j}\sigma_{j},\quad  Y_{j} \triangleq \sum_{j'} x_{j'} Q_{j,j'} \sigma'_{j'} 
\end{align}

\begin{proposition}\label{prop:reduction}
For $Y_j$ defined as in \Cref{eq:double_sum} we have
\begin{align}
\mathbf{E} \exp\left(t\sum_{j} Y_{j} x_{j} \sigma_{j}\right)  \leqslant \mathbf{E}_{}\exp(t^2/2 \cdot \sum_j x_j^2 Y_j^2 )
\end{align}
\end{proposition}
\begin{proof}
We have the following chain of estimates
\begin{align*}
\mathbf{E}\left[ \left. \exp\left(t\sum_{j} Y_{j} x_{j} \sigma_{j}\right) \right| Y_j \right] & \leqslant^{(a)} \mathbf{E}_{(\sigma_j)_j}\left. \left[ 
\exp(t\sum_j Y_j x_j \sigma_j) \right| Y_j \right] \\ 
&=^{(b)} \prod_j \mathbf{E}_{\sigma_j} \exp(t x_j Y_j \sigma_j) \\
& \leqslant^{(c)} \prod_j \exp(t^2 x_j^2 Y_j^2 / 2) = \exp(t^2/2\cdot \sum_j x_j^2  Y_j^2)
\end{align*}
Here (a) follows from the fact that when conditioning on fixed value of $Y_j$ the only remained randomness is that of $\sigma_j$ ( indepedence of $Y_j$). Equality (b) follows as $\sigma_j$ are independent, and (c) because of Hoeffding's inequality.
Finally we take expectation over $Y_j$.
\end{proof}

\begin{proposition}\label{prop:extreme}
For any (possibly correlated) r.vs $Y_j$ and unit vector $x$ it holds that
\begin{align}
\mathbf{E}\exp(t^2/2 \cdot \sum_j x_j^2 Y_j^2) \leqslant \max_j \mathbf{E}\exp(t^2/2\cdot Y_j^2)
\end{align}
\end{proposition}
\begin{proof}
We have
\begin{align}
\mathbf{E}\exp(t^2/2 \cdot \sum_j x_j^2 Y_j^2) & \leqslant^{(a)}  \sum_{j} x_j^2 \mathbf{E}\exp(t^2/2\cdot Y_j) \\
& \leqslant^{(b)} \max_j \mathbf{E}\exp(t^2 Y_j^2/2)
\end{align}
where (a) follows by Jensen's inequality with weights $x_j^2$ (they are valid weights due to the assumption $\|x\|_2=1$) and (b) follows because $\sum_j x_j^2=1$.
\end{proof}
This discussion can be summarized as follows
\begin{corollary}\label{cor:mgf_square_linear} We have the following bound
\begin{align}
\mathbf{E}\exp(t E(x)) \leqslant \max_j\mathbf{E}\exp(8 t^2 Y_j^2)
\end{align}
\end{corollary}

\subsection{Sub-Gaussianity of Linear Form}\label{sec:sub-gauss_linear_form}

In view of \Cref{cor:mgf_square_linear} we are left with the task of upper-bounding $\mathbf{E}\exp(t^2 Y_j^2)$. 
To this end we estimate the MGF and hence the tail of $Y_j$.

\begin{proposition}[Sub-gaussian norm]\label{prop:lin_subgauss}
For every $j$ define $v_j=\mathbf{E}(\sum_{j'} Q_{j,j'}^2 x_{j'}^2)^{1/2}$, then
\begin{align*}
\mathbf{E}\exp(t Y_j) \leqslant \exp(t^2 v_j^2/2), \quad t\in\mathbb{R}.
\end{align*}
\end{proposition}
\begin{proof}
For every $j$ we have
\begin{align*}
\|   Y_j  \|_{sG} & \leqslant^{(a)} \mathbf{E}_{(Q_{j,j'})_{j'}} \|  \mathbf{E}[ Y_j | (Q_{j,j'})_{j'}] \|_{sG} \\
& =^{(b)} \mathbf{E}_{(Q_{j,j'})_{j'}}   \| \mathbf{E}[ \sum_{j'} Q_{j,j'} x_{j'} \sigma'_{j'} | (Q_{j,j'})_{j'}]  \|_{sG} \\
& =^{(c)} \mathbf{E}_{(Q_{j,j'})_{j'}}   ( \sum_j x_j^2 Q_{j,j'}^2 \| \sigma'_{j'}  \|_{\psi_2})^{1/2} \\
& =^{(d)} \mathbf{E} \left(\sum_{j'} x_{j'}^2 Q_{j,j'}^2\right)^{1/2} 
\end{align*}
Here (a) follows because t$\|\cdot\|_{sG}$ is a norm and hence convex. Equality (b) uses the explicit form of $Y_{j}$.
Then (c) holds due to~\Cref{lemma:subg_norm} because conditioned on $Q_{j,j'}$ for all $j'$ we are left with linear combination of independent Rademachers $\sigma'_j$. Then (d) holds because the $\|\cdot\|_{sG}$ norm of a Rademacher distribution is bounded by 1.
\end{proof}

\subsection{Bounding MGF of Quadratic Form}

Define $v \triangleq \sup_{x:\|x\|_2\leqslant 1} \max_j \mathbf{E}(\sum_{j'} Q_{j,j'}^2 x_{j'}^2)^{1/2}$
as in \Cref{thm:our_main}. We clearly have $v_j \leqslant v$ where $v_j$ are defined in \Cref{prop:lin_subgauss}.
Due to \Cref{cor:mgf_square_linear} we obtain
\begin{proposition}\label{eq:mgf_quadform}
Let $v$ be as above, then
\begin{align*}
\mathbf{E}\exp(t E(x)) \leqslant \exp\left(\frac{ t^2\cdot 16v^2}{1- t^2\cdot 16 v^2}\right), \quad 4 t v < 1.
\end{align*}
\end{proposition}
\begin{proof}
As noticed $Y_j$ are sub-gaussian with factor $v$.
\Cref{lemma:subgauss_square} with $t$ replaced by $8t^2$ gives
\begin{align*}
\mathbf{E}\exp(8t^2 Y_j^2) \leqslant \exp\left(\frac{ t^2\cdot 16 v^2}{1- t^2\cdot 16 v^2}\right), \quad 16 t^2 v^2 < 1.
\end{align*}
The claim now follows by \Cref{cor:mgf_square_linear}.
\end{proof}
In some cases (as our version of JL Lemma) the bound for $v$ simplifies even further
\begin{proposition}\label{prop:simpler_norm}
If $Q_{j,j'}$ are normalized so that $\mathbf{E}Q_{j,j'}^2 \leqslant q^2$ for all $j,j'$ then $v^2 \leqslant q^2$.
\end{proposition}
\begin{proof}
By Jensen's inequality applied to $u\to u^{1/2}$ (concave!)
we have $\mathbf{E}(\sum_{j'} Q_{j,j'}^2 x_{j'}^2)^{1/2} \leqslant (\sum_{j'} \mathbf{E} Q_{j,j'}^2 x_{j'}^2)^{1/2}$.
The result now follows because $\mathbf{E}Q_{j,j'}^2\leqslant q^2$ and $\sum_{j'}x_{j'}^2=1$.
\end{proof}

\subsection{Bounding Tail of Quadratic Form}

Having bounded the MGF we easily obtain the tail bound for $E(x)$.

\begin{corollary}\label{cor:tail_quadform}
We have that $E(x)$ is sub-gamma with parameters $(32v^2)^{1/2}$ and $4v$. In particular we have the tail of
\begin{align}
\Pr[E(x)>\epsilon]\leqslant \exp(-\min(\epsilon^2/ 128 v^2,\epsilon/16 v))
\end{align}
\end{corollary}
\begin{proof}
The bound \Cref{eq:mgf_quadform} can be further upper-bonded by $\exp\left(\frac{ t^2\cdot 16v^2}{1- t\cdot 4 v}\right)$
because $4 t v < 1$ implies $16 t^2 v^2 < 4tv$. By \Cref{def:sub-gamma} we conclude that $E(x)$ is sub-gamma with parameters
$v:= (32v^2)^{1/2}$ and $c:=4v$. The tail bound follows by \Cref{lemma:sugbamma_tails}.
\end{proof}

\begin{corollary}[Concluding main result]
\Cref{thm:our_main} holds.
\end{corollary}
\begin{proof}
Follows directly by \Cref{cor:tail_quadform} and \Cref{prop:simpler_norm} (the simpler bound for $v$).
\end{proof}

\section{Conclusion}

We have discussed a simpler proof of JL Lemma with neuroscience-based constraints.
The proof uses only basic estimates on moment generating functions (sub-gaussian and sub-gamma type), and offers explicit constant.

\bibliography{citations}

\end{document}